\documentclass{amsart}
\usepackage{amssymb, amsmath, mathrsfs, verbatim, url, bbm}

\theoremstyle{plain}
\newtheorem{theorem}{Theorem}
\newtheorem{lemma}[theorem]{Lemma}
\newtheorem{proposition}[theorem]{Proposition}
\newtheorem{corollary}[theorem]{Corollary}

\theoremstyle{definition}
\newtheorem{definition}[theorem]{Definition}
\newtheorem{question}{Question}

\newcommand{\diam}{\mathsf{diam}}

\newcommand{\cl}{\mathsf{cl}}
\newcommand{\dom}{\mathsf{dom}}

\newcommand{\supp}{\mathsf{supp}}
\newcommand{\bG}{\mathbf{\Gamma}}
\newcommand{\PSP}{\mathsf{PSP(\mathbf{\Gamma})}}
\newcommand{\kPSP}{\kappa\text{-}\mathsf{PSP(\mathbf{\Gamma})}}
\newcommand{\Pco}{\mathsf{PSP(coanalytic)}}
\newcommand{\Po}{\mathsf{PSP(open)}}
\newcommand{\Pc}{\mathsf{PSP(closed)}}
\newcommand{\Pcl}{\mathsf{PSP(clopen)}}
\newcommand{\Pf}{\mathsf{PSP(F_\sigma)}}
\newcommand{\Pg}{\mathsf{PSP(G_\delta)}}
\newcommand{\Pa}{\mathsf{PSP(analytic)}}
\newcommand{\WO}{\mathsf{WO}}
\newcommand{\CH}{\mathsf{CH}}
\newcommand{\MA}{\mathsf{MA}}
\newcommand{\VL}{\mathsf{V=L}}
\newcommand{\GP}{\mathsf{GP}}
\newcommand{\klGP}{\mathsf{(\kappa,\lambda)}\textrm{-}\mathsf{GP}}
\newcommand{\ZFC}{\mathsf{ZFC}}
\newcommand{\COF}{\mathsf{Cof}}

\newcommand{\Aa}{\mathcal{A}}

\newcommand{\CC}{\mathcal{C}}
\newcommand{\DD}{\mathcal{D}}
\newcommand{\TT}{\mathcal{T}}
\newcommand{\Ss}{\mathcal{S}}

\newcommand{\FF}{\mathcal{F}}

\newcommand{\II}{\mathcal{I}}
\newcommand{\JJ}{\mathcal{J}}

\newcommand{\UU}{\mathcal{U}}
\newcommand{\VV}{\mathcal{V}}
\newcommand{\PP}{\mathcal{P}}

\newcommand{\PPP}{\mathbb{P}}

\newcommand{\cccc}{\mathfrak{c}}

\newcommand{\bbbb}{\mathfrak{b}}

\begin{document}

\title[Distinguishing perfect set properties]{Distinguishing perfect set
properties in separable metrizable spaces}

\author{Andrea Medini}
\address{Kurt G\"odel Research Center for Mathematical Logic
\newline\indent University of Vienna, W\"ahringer Stra{\ss}e 25, A-1090 Wien,
Austria}
\email{andrea.medini@univie.ac.at}
\urladdr{http://www.logic.univie.ac.at/\~{}medinia2/}

\date{August 9, 2014}

\thanks{The author acknowledges the support of the FWF grant I 1209-N25.}

\begin{abstract}
All spaces are assumed to be separable and metrizable. Our main result is that
the statement ``For every space $X$, every closed subset of $X$ has the perfect
set property if and only if every analytic subset of $X$ has the perfect set
property'' is equivalent to $\bbbb>\omega_1$ (hence, in particular, it is independent of $\ZFC$).
This, together with a theorem of Solecki and an example of Miller, will allow us to
determine the status of the statement ``For every space $X$, if every $\bG$
subset of $X$ has the perfect
set property then every $\bG'$ subset of $X$ has the perfect set
property'' as $\bG,\bG'$ range over all pointclasses of complexity at most
analytic or coanalytic.

Along the way, we define and investigate a property of independent interest. We
will say that a subset $W$ of $2^\omega$ has the \emph{Grinzing property} if it
is uncountable and for every uncountable $Y\subseteq W$ there exists an
uncountable collection consisting of uncountable subsets of $Y$ with pairwise
disjoint closures in $2^\omega$. The following theorems hold.
\begin{enumerate}
\item There exists a subset of $2^\omega$ with the Grinzing property.
\item Assume $\MA+\neg\CH$. Then $2^\omega$ has the Grinzing property.
\item Assume $\CH$. Then $2^\omega$ does not have the Grinzing property.
\end{enumerate}
The first result was obtained by Miller using a theorem of Todor\v{c}evi\'c, and is needed in the proof of our main result.
\end{abstract}

\maketitle

All spaces are assumed to be separable and metrizable, but not necessarily Polish. See Section 1 for notation and terminology.
\begin{definition}
Let $X$ be a space and $\bG$ a pointclass. We will say that $X$ has the
\emph{perfect set property for $\bG$ subsets} (briefly, the $\PSP$)
if every $\bG$ subset of $X$ is either countable or it contains a copy of
$2^\omega$.
\end{definition}
\noindent Notice that the $\PSP$ gets stronger as $\bG$
gets bigger.

One of the classical problems of descriptive set theory consists in determining
for which pointclasses $\bG$ the statement ``Every Polish space has the
$\PSP$'' holds. The following three famous theorems essentially
solve this
problem (see \cite[Theorem 12.2(c)]{kanamori}, \cite[Theorem 13.12]{kanamori}, and \cite[Proposition 27.5]{kanamori} respectively).
\begin{theorem}[Suslin]\label{suslin}
Every Polish space has the $\Pa$.
\end{theorem}
\begin{theorem}[G\"odel]\label{godel}
Assume $\VL$. Then no uncountable Polish space has the $\Pco$.
\end{theorem}
\begin{theorem}[Davis]
Assume the axiom of Projective Determinacy. Then every Polish space has the
$\mathsf{PSP(projective)}$.\footnote{Since the axiom of Projective Determinacy has a fairly high consistency strength, it might be worth noting that an inaccessible cardinal is enough to obtain the consistency of ``Every Polish space has the $\mathsf{PSP(projective)}$'' (this is due to Solovay, see \cite[Theorem 26.14(ii)]{jech}).}
\end{theorem}

But what happens in arbitrary (that is, not necessarily Polish) spaces? By the following simple proposition, the problem described above becomes trivial.
\begin{proposition}
Let $X$ be a Bernstein set in some uncountable Polish space. Then $X$ does not have the $\PSP$ for any poinclass $\bG$.
\end{proposition}
\begin{proof}
Let $\bG$ be a pointclass. Then $X$ itself is an uncountable $\bG$ subset of $X$ that does not contain any copy of $2^\omega$.
\end{proof}

Much less trivial, however, is to determine exactly for which pointclasses $\bG,\bG'$ the $\PSP$ implies the $\mathsf{PSP(\bG')}$. More precisely, we will investigate the status of the statement ``For every space $X$, if $X$ has the $\PSP$ then $X$ has the $\mathsf{PSP(\bG')}$'' as $\bG,\bG'$ range over all pointclasses of complexity at most analytic or coanalytic. The reason for this limitation on the complexity lies in Theorem \ref{suslin} and Theorem \ref{godel}, but the investigation could be carried further (see Section 6).

While Theorem \ref{completepicture} gives a complete solution to this problem, we will focus first on a special case, which turned out to be the most interesting one. More specifically, we will consider the statement ``For every space $X$, if $X$ has the $\Pc$ then $X$ has the $\Pa$''. As we will show in Section 3, this statement is in fact equivalent to $\bbbb>\omega_1$ (see Theorem \ref{main}). We will need an auxiliary property, which is defined and investigated in Section 2. In Section 4, we will complement Theorem \ref{main} by establishing all results relevant to our problem that can be proved in $\ZFC$ alone. Finally, in Section 5 and Section 6, we will investigate natural generalizations of the properties that we considered and state several open questions.

To the reader interested in more instances of ``dropping Polishness'', we recommend the articles \cite{millerl} and \cite{millerp} of Miller, which partly inspired our work.

\section{Terminology and notation}

We will generally follow \cite{kechris}. However, since the spaces that we will
deal with are not necessarily Polish, we prefer to recall the relevant notions
and a few basic facts about them. Throughout this article, $\bG$ will always
denote one of the following (boldface) pointclasses.
\begin{itemize}
\item $\mathbf{\Sigma}^0_\xi$ or $\mathbf{\Pi}^0_\xi$, where $\xi$ is an ordinal
such that $1\leq\xi <\omega_1$ (these are the \emph{Borel pointclasses}).
\item $\mathbf{\Sigma}^1_n$ or $\mathbf{\Pi}^1_n$, where $n$ is an ordinal such
that $1\leq n <\omega$ (these are the \emph{projective pointclasses}).
\end{itemize}
We will assume that the definition of a $\bG$ subset of a Polish space is
well-known, and recall that it can be generalized to arbitrary spaces as
follows. In the case of the Borel pointclasses, this is not really necessary,
because the usual definition works in arbitrary spaces (see for example \cite[Section
11.B]{kechris}). However, we prefer to give a unified treatment.
\begin{definition}
Fix a pointclass $\bG$. Let $X$ be a space. We will say that $A\subseteq X$ is a
\emph{$\bG$ subset of $X$} if there exists a Polish space $T$ containing $X$ as
a subspace such that $A=B\cap X$ for some $\bG$ subset $B$ of $T$.
\end{definition}
\noindent It is easy to check that, when $\bG$ is a Borel pointclass, the above
definition coincides with the usual one. We will denote $\mathbf{\Sigma}^0_1$ by
$\mathsf{open}$, $\mathbf{\Pi}^0_1$ by $\mathsf{closed}$, $\mathbf{\Sigma}^0_2$
by $\mathsf{F_\sigma}$, $\mathbf{\Pi}^0_2$ by $\mathsf{G_\delta}$,
$\mathbf{\Sigma}^1_1$ by $\mathsf{analytic}$, $\mathbf{\Pi}^1_1$ by
$\mathsf{coanalytic}$. We will also denote $\bigcup_{1\leq\xi<
\omega_1}\mathbf{\Sigma}^0_\xi$ by $\mathsf{Borel}$ and $\bigcup_{1\leq n<
\omega}\mathbf{\Sigma}^1_n$ by $\mathsf{projective}$.

The following ``reassuring'' proposition can be proved by induction on $\bG$.
Since every space is homeomorphic to a subspace of $[0,1]^\omega$, one can
assume without loss of generality that $T$ is Polish. This makes it possible to
use Lavrentiev's Theorem (see \cite[Theorem 3.9]{kechris}) in the inductive
step from $\mathbf{\Pi}^1_n$ to $\mathbf{\Sigma}^1_{n+1}$, which is the only
non-trivial part of the proof.
\begin{proposition}
Fix a pointclass $\bG$. Let $X$ be a space and $A$ a subset of $X$. The
following conditions are equivalent.
\begin{itemize}
\item\label{some} $A$ is a $\bG$ subset of $X$.
\item\label{every} For every space $T$ containing $X$ as a subspace there exists
a $\bG$ subset $B$ of $T$ such that $A=B\cap X$.
\end{itemize}
\end{proposition}

One could also define the so-called \emph{ambiguous pointclasses} as follows.
\begin{itemize}
\item Let $\mathbf{\Delta}^0_\xi=\mathbf{\Sigma}^0_\xi\cap\mathbf{\Pi}^0_\xi$
for every ordinal $\xi$ such that $1\leq\xi <\omega_1$.
\item Let $\mathbf{\Delta}^1_n=\mathbf{\Sigma}^1_n\cap\mathbf{\Pi}^1_n$ for
every ordinal $n$ such that $1\leq n <\omega$.
\end{itemize}
However, it is not hard to realize that they
would not add any interesting content to our results, while making the
exposition more cumbersome. Therefore, we will exclude them from our discussion.
The only exception is $\mathbf{\Delta}^0_1$, which we will denote by
$\mathsf{clopen}$ (see Proposition \ref{twodimbernstein} and Theorem
\ref{completepicture}).

Our reference for topology will be \cite{vanmillt}. By $X\approx Y$ we will mean that the spaces $X$ and $Y$ are homeomorphic. In this case, we will say that $Y$
is a \emph{copy} of $X$. Given an infinite cardinal $\kappa$, we will say that a
space $X$ is \emph{$\kappa$-crowded} if it is non-empty and every non-empty
open subset of $X$ has size at least $\kappa$. Recall that a sequence
$\langle A_n:n\in\omega\rangle$ of subsets of a space $X$ \emph{converges}
to a point $x\in X$ if for every neighborhood $U$ of $x$ there
exists $m\in\omega$ such that $A_n\subseteq U$ whenever $n\geq m$.

Recall that a subset $B$ of an uncountable Polish space
$T$ is a \emph{Bernstein set} if $B\cap K\neq\varnothing$ and $(T\setminus
B)\cap K\neq\varnothing$ for every copy $K$ of $2^\omega$ in $T$. It is easy to
see that Bernstein sets exist in $\ZFC$. Since $2^\omega\approx 2^\omega\times
2^\omega$, every Bernstein set has size $\cccc$.

Recall that a subset $L$ of a Polish space $T$ is a \emph{Luzin
set} if it is uncountable and every uncountable subset of $L$ is non-meager in $T$. It is easy to see that Luzin sets exist under the assumption $\mathsf{cof(BP)}=\omega_1$.

Our reference for pure set theory will be \cite{jech}. Given $f,g\in\omega^\omega$, we will write $f<^\ast g$ if there exists $m\in\omega$ such that $f(n)<g(n)$ for every $n\geq m$. Recall that a family $B\subseteq \omega^\omega$ is \emph{unbounded} if there exists no $g\in\omega^\omega$ such that $f<^\ast g$ for every $f\in B$. We will denote by $\bbbb$ the least size of an unbounded family in $\omega^\omega$.

\section{The Grinzing property}

In this section we investigate the property defined as follows.
The only result
from this section that will be used in the remainder of the article (namely, in
the proof of Proposition \ref{bsmall}) is Corollary \ref{existsgrinzing}.
However, we believe that Proposition \ref{gpch} and Theorem \ref{gpma} are of
independent interest. See also Section 5.

\begin{definition}
We will say that a subset $W$ of $2^\omega$ has
the \emph{Grinzing property}\footnote{This property is named after a lovely area of Vienna, in which the author lived while this article was being written.}
(briefly, the $\GP$)
if it is uncountable and for every uncountable $Y\subseteq W$ there exist
uncountable subsets $Y_\alpha$ of $Y$ for $\alpha\in\omega_1$
such that $\cl(Y_\alpha)\cap \cl(Y_\beta)=\varnothing$ whenever
$\alpha\neq\beta$, where the closure is taken in $2^\omega$.
\end{definition}
\noindent Notice that an uncountable subset $W$ of $2^\omega$ has the $\GP$ if and only if every subset of $W$ of size $\omega_1$ has the $\GP$.

An early version of this article contained an erroneous proof of Corollary \ref{existsgrinzing}. Miller spotted the mistake and supplied a new proof, based on Theorem \ref{todorcevic}. We will need the following definition. Given a subtree $\TT$ of $\omega^{<\omega_1}$ and  $\alpha<\omega_1$, let $\TT_\alpha=\TT\cap\omega^\alpha$ denote the $\alpha$-th level of $\TT$. Also let $\TT_{<\gamma}=\bigcup_{\alpha<\gamma}\TT_\alpha$ for every limit ordinal $\gamma<\omega_1$.

\begin{definition}[Todor\v{c}evi\'c]\label{aronszajn}
We will say that $\langle K_s:s\in\TT\rangle$ is an \emph{Aronszajn tree of perfect sets} if the following conditions are satisfied.
\begin{enumerate}
\item $\TT$ is a subtree of $\omega^{<\omega_1}$.
\item $\TT_\alpha$ is non-empty and countable for every $\alpha\in\omega_1$.
\item $K_s$ is a perfect subset of $2^\omega$ for every $s\in\TT$.
\item $K_t\subsetneq K_s$ whenever $t\supsetneq s$.
\item $K_s\cap K_t=\varnothing$ whenever $t\perp s$.
\end{enumerate}
\end{definition}
\noindent Using the fact that there are no strictly decreasing $\omega_1$-sequences of closed subsets of $2^\omega$, it is easy to check that if $\langle K_s:s\in\TT\rangle$ is an Aronszajn tree of perfect sets then $\TT$ is in fact an Aronszajn tree in the usual sense.

The proof of following result is taken from the first part of the proof of \cite[Theorem 4]{galvinmiller}, and it is included for the sake of completeness. We will employ the well-known technique of fusion, for which we refer to \cite[pages 244-245]{jech}. When using the ordering $\leq_n$, we will identify a perfect subset $K$ of $2^\omega$ with the corresponding subtree $p=\{x\upharpoonright n:x\in K, n\in\omega\}$ of $2^{<\omega}$.

\begin{theorem}[Todor\v{c}evi\'c]\label{todorcevic}
There exists an Aronszajn tree of perfect sets.
\end{theorem}
\begin{proof}
We will construct $\TT_\alpha$ and $K_s$ for $s\in\TT_\alpha$ simultaneously by transfinite recursion, making sure that conditions (1)-(5) in Definition \ref{aronszajn} are satisfied. Furthermore, we will make sure that the following additional condition is satisfied.
\begin{enumerate}
\item[(6)] Whenever $\beta < \alpha$ and $n\in\omega$, given any $s\in\TT_\beta$, there exists $t\in\TT_\alpha$ such that $t\supseteq s$ and $K_t\leq_n K_s$.
\end{enumerate}
\noindent Start by setting $\TT_0=\{\varnothing\}$ and $K_\varnothing=2^\omega$. The construction at a successor stage is straightforward, so we leave it to the reader. Now assume that $\gamma<\omega_1$ is a limit ordinal, and that $\TT_{<\gamma}$ and $K_s$ for $s\in\TT_{<\gamma}$ have already been constructed.

Let $\langle \alpha_\ell:\ell\in\omega\rangle$ be a strictly increasing sequence consisting of elements of $\gamma$ such that $\alpha_\ell\to\gamma$. Fix $s\in\TT_{\alpha_n}$. We will inductively define a chain $\langle t_\ell(s):\ell\geq n\rangle$ such that $t_\ell(s)\in\TT_{\alpha_\ell}$ for every $\ell\geq n$, using the inductive hypothesis (6). Start by letting $t_n(s)=s$. Given $t_\ell(s)$ for some $\ell\geq n$, choose $t_{\ell+1}(s)\in\TT_{\alpha_{\ell+1}}$ such that $t_{\ell+1}(s)\supseteq t_\ell(s)$ and $K_{t_{\ell+1}(s)}\leq_{\ell+1} K_{t_\ell(s)}$. Now let $t(s)=\bigcup_{\ell\geq n} t_\ell(s)$.

Define $\TT_\gamma=\{t(s):s\in\bigcup_{\ell\in\omega}\TT_{\alpha_\ell}\}$. For each $t\in\TT_\gamma$, fix $s\in\TT_{\alpha_n}$ such that $t=t(s)$, then define $K_t=\bigcap_{\ell\geq n}K_{t_\ell(s)}$. Notice that $K_t$ is a perfect subset of $2^\omega$ by fusion. Furthermore, it is clear that this definition does not depend on the choice of $s$. Finally, it is easy to check that conditions (1)-(6) will be mantained.
\end{proof}

\begin{corollary}[Miller]\label{existsgrinzing}
There exists a subset of $2^\omega$ with the $\GP$.
\end{corollary}
\begin{proof}
Let $\langle K_s:s\in\TT\rangle$ be an Aronszajn tree of perfect sets.
Pick $w_s\in K_s$ for every $s\in\TT$ so that $w_s\neq w_t$ whenever $s\neq t$. Let $W=\{w_s:s\in\TT\}$. We will show that $W$ has the $\GP$. Fix an uncountable $Y\subseteq W$ and let $\Ss$ be the subtree of $\TT$ generated by $\{s\in\TT:w_s\in Y\}$. Assume without loss of generality that $\Ss$ is normal (that is, $\{t\in\Ss: t\supseteq s\}$ is uncountable for every $s\in \Ss$).

Notice that $\langle K_s:s\in\Ss\rangle$ is still an Aronszajn tree of perfect sets. It follows that $\Ss$ is not ccc, otherwise forcing with $\Ss$ would yield a strictly descending $\omega_1$-sequence of closed subsets of $2^\omega$, contradicting the fact that $\omega_1$ is preserved. So we can fix an uncountable antichain $\langle s_\alpha:\alpha\in\omega_1\rangle$ in $\Ss$. Set $Y_\alpha=Y\cap K_{s_\alpha}$ for $\alpha\in\omega_1$. Obviously $\cl(Y_\alpha)\cap \cl(Y_\beta)=\varnothing$ whenever $\alpha\neq\beta$, where the closure is taken in $2^\omega$. Furthermore, the fact that $\Ss$ is normal implies that each $Y_\alpha$ is uncountable.
\end{proof}

The following two results show that the statement ``$2^\omega$ has the $\GP$'' is independent of $\ZFC$.
\begin{proposition}\label{gpch}
Assume $\CH$. Then $2^\omega$ does not have the $\GP$.
\end{proposition}
\begin{proof}
We will give two proofs. The first proof uses the fact that $\CH$ implies
$\bbbb=\omega_1$. In particular, we can fix an unbounded family
$B=\{f_\alpha:\alpha\in\omega_1\}\subseteq\omega^\omega$ consisting of increasing functions such that $f_\alpha <^\ast f_\beta$ whenever $\alpha<\beta$. Let $T=(\omega +1)^\omega$. Notice that
$B\subseteq\omega^\omega\subseteq T\approx 2^\omega$.
The uncountable set $B$ will witness that $T$ does not have the $\GP$. Let
$$
T_\infty=\{s^{\frown}\langle\omega,\omega,
\ldots\rangle:s\in\omega^{<\omega}\}\subseteq T.
$$
Since $T_\infty$ is countable, it will be enough to show that $\cl(C)\cap
T_\infty\neq\varnothing$ whenever $C\subseteq B$ is uncountable, where the
closure is taken in $T$. Since every uncountable subset of $B$ is unbounded, this follows from Lemma \ref{overspill} and Lemma \ref{closureincreasing}.

The second proof uses the fact that $\CH$ implies
$\mathsf{cof(BP)}=\omega_1$.
In particular, we can fix a Luzin set $L$ in $2^\omega$. The uncountable set $L$
will witness that $2^\omega$ does not have the $\GP$.
Assume that $L_\alpha$ is an uncountable subset of $L$ for every
$\alpha\in\omega_1$.
By the definition of Luzin set, no $L_\alpha$ can be nowhere dense. Therefore,
there exist non-empty open
subsets $U_\alpha$ of $2^\omega$ for $\alpha\in\omega_1$ such that
$U_\alpha\subseteq\cl(L_\alpha)$ for each $\alpha$.
Let $\alpha\neq\beta$ be such that $U_\alpha\cap U_\beta\neq\varnothing$. It is
clear that $\cl(L_\alpha)\cap\cl(L_\beta)\neq\varnothing$.
\end{proof}

\begin{theorem}\label{gpma}
Assume $\MA+\neg\CH$. Then $2^\omega$ has the $\GP$.
\end{theorem}
\begin{proof}
Fix $Y\subseteq 2^\omega$ of size $\omega_1$. Let $Y=\{y_\xi:\xi\in\omega_1\}$
be an injective enumeration.
Fix uncountable subsets $Z_\alpha$ of $Y$ for $\alpha\in\omega_1$ such that
$Z_\alpha\cap Z_\beta=\varnothing$ whenever $\alpha\neq\beta$.
By Lemma \ref{findcrowded}, we can also assume that each $Z_\alpha$ is
$\omega_1$-crowded.

Let $\PPP$ be the set of all pairs of the form $p=\langle n,F\rangle=\langle
n^p,F^p\rangle$, where $n
\in\omega$ and $F=\langle F_\alpha:\alpha\in\omega_1\rangle=\langle
F_\alpha^p:\alpha\in\omega_1\rangle$
satisfy the following conditions.
\begin{enumerate}
\item Each $F_\alpha$ is a finite subset of $Z_\alpha$.
\item $\supp(p)=\{\alpha\in\omega_1:F_\alpha\neq\varnothing\}$ is finite.
\item\label{separateFs} $\{x\upharpoonright n:x\in
F_\alpha\}\cap\{x\upharpoonright n:x\in F_\beta\}=\varnothing$ whenever
$\alpha\neq\beta$.
\end{enumerate}
Given $p,q\in\PPP$, define $q\leq p$ if the following conditions are satisfied.
\begin{enumerate}
\item[(4)] $n^q\geq n^p$.
\item[(5)] $F_\alpha^q\supseteq F_\alpha^p$ for each $\alpha$.
\item[(6)] $\{x\upharpoonright n^p:x\in
F_\alpha^q\}=\{x\upharpoonright n^p:x\in F_\alpha^p\}$ whenever
$\alpha\in\supp(p)$.
\end{enumerate}

For any given $\delta,\eta\in\omega_1$, define
$$
D_{\delta,\eta}=\{p\in\PPP:y_\xi\in F_\delta^p\textrm{ for some
}\xi\geq\eta\}.
$$
In order to show that each $D_{\delta,\eta}$ is dense, fix
$p\in\PPP$ and
$\delta,\eta\in\omega_1$. Since the case $\delta\notin\supp(p)$ is
trivial, we will assume that $\delta\in\supp(p)$. Fix $x\in F_\delta^p$. Notice
that $Z_\delta\cap
[x\upharpoonright n^p]$ is uncountable because $Z_\delta$ is $\omega_1$-crowded
and $x\in Z_\delta$.
In particular, there exists $\xi\geq\eta$ such that
$y_\xi\in Z_\delta$ and $y_\xi\upharpoonright n^p=x\upharpoonright n^p$. Let
$F^q=\langle F_\alpha^q:\alpha\in\omega_1\rangle$, where $F_\alpha^q=F_\alpha^p$
for every
$\alpha\neq\delta$ and $F_\delta^q=F_\delta^p\cup\{y_\xi\}$. Define $q=\langle
n^q,F^q\rangle$,
where $n^q\geq n^p$ is such that condition (\ref{separateFs}) is satisfied. It
is clear that $q\in D_{\delta,\eta}$ and $q\leq p$. 

Next, we will show that $\PPP$ is ccc. Fix $\Aa\subseteq\PPP$ such that
$|\Aa|=\omega_1$. By passing to an uncountable subset of $\Aa$, we can assume
that there exists $n\in\omega$ such that $n^p=n$ for all $p\in\Aa$. By the
Delta System Lemma, we can assume that there exists a finite
$R\subseteq\omega_1$ such that $\supp(p)\cap\supp(q)=R$ whenever
$p,q\in\Aa$ are distinct. By passing to an uncountable subset of $\Aa$ once
more, we can assume that $\{x\upharpoonright n:x\in
F_\alpha^p\}=\{x\upharpoonright n:x\in F_\alpha^q\}$ for every $\alpha\in R$
whenever $p,q\in\Aa$. Now fix distinct $p,q\in\Aa$ and let $F^r=\langle
F_\alpha^p\cup
F_\alpha^q:\alpha\in\omega_1\rangle$. Define $r=\langle n^r,F^r\rangle$, where
$n^r\geq n$ is such
that condition (\ref{separateFs}) is satisfied. It is clear that $r\in\PPP$ and
$r\leq p,q$. In particular, $\Aa$ is not an antichain.

Since we are assuming $\MA+\neg\CH$ and the collection of dense sets
$$
\DD=\{D_{\delta,\eta}:\delta,\eta\in\omega_1\}
$$
has size $\omega_1$, there exists a $\DD$-generic filter $G\subseteq\PPP$.
Define
$Y_\alpha=\bigcup\{F_\alpha^p:p\in G\}$ for every $\alpha\in\omega_1$. The dense
sets $D_{\delta,\eta}$ ensure that each $Y_\alpha$ is uncountable.

Finally, in order to get a contradiction, assume that
$w\in\cl(Y_\alpha)\cap\cl(Y_\beta)$, where $\alpha\neq\beta$. Fix $p\in G$ such
that $\{\alpha,\beta\}\subseteq\supp(p)$ and let $n=n^p$. Let $w_\alpha\in
Y_\alpha\cap [w\upharpoonright n]$ and $w_\beta\in Y_\beta\cap [w\upharpoonright
n]$. There exist $q_\alpha, q_\beta\in G$ such that $w_\alpha\in
F_\alpha^{q_\alpha}$ and $w_\beta\in F_\beta^{q_\beta}$.
Without loss of generality, assume that $q_\alpha, q_\beta\leq p$. In
particular, by condition (6), there exist $z_\alpha\in F_\alpha^p$ and
$z_\beta\in
F_\beta^p$ such that $z_\alpha\upharpoonright n=w_\alpha\upharpoonright n$ and
$z_\beta\upharpoonright n=w_\beta\upharpoonright n$. But this implies
$$
z_\alpha\upharpoonright n=w_\alpha\upharpoonright n=w\upharpoonright
n=w_\beta\upharpoonright n=z_\beta\upharpoonright n,
$$
which contradicts condition (\ref{separateFs}).
\end{proof}

\section{The main result}

The aim of this section is to prove the following result, which is an immediate
consequence of Proposition \ref{bbig} and Proposition \ref{bsmall}.

\begin{theorem}\label{main}
The following are equivalent.
\begin{itemize}
\item $\bbbb>\omega_1$.
\item For every space $X$, if $X$ has the $\Pc$ then $X$ has the $\Pa$.
\end{itemize}
\end{theorem}

\begin{proposition}\label{bbig}
Assume $\bbbb>\omega_1$. Let $X$ be a space with the $\Pc$. Then $X$ has the
$\Pa$.\end{proposition}
\begin{proof}
Assume without loss of generality that $X$ is a subspace of $T=[0,1]^\omega$. Fix an analytic subset $A$ of $T$ such that $A\cap X$ is uncountable. We will show that $A\cap X$ contains a copy of $2^\omega$. Since $A$ is analytic, there exists a closed subset $C$ of
$T\times\omega^\omega$ such that
$\pi[C]=A$, where $\pi: T\times\omega^\omega\longrightarrow
T$ is the projection on the first coordinate. Since $A\cap X$ is uncountable, it is possible to fix $Z\subseteq\omega^\omega$ such that
$|Z|\leq\omega_1$ and $\pi[(T\times Z)\cap C]\cap X$ is uncountable.

By the assumption $\bbbb>\omega_1$, there exist compact subsets $K_n$ of $\omega^\omega$ for
$n\in\omega$ such that $Z\subseteq\bigcup_{n\in\omega}K_n$. Observe that each $\pi[(T\times K_n)\cap C]\cap X$ is a closed subset of $X$
contained in $A\cap X$.
Furthermore, since $\omega_1$ has uncountable cofinality, at least one of them
must be uncountable. The fact that $X$ has the $\Pc$ concludes the proof.
\end{proof}

In the remainder of this section we will
employ the following three lemmas, which can be safely considered folklore.

\begin{lemma}\label{noperfect}
Assume that $B=\{f_\alpha:\alpha\in\omega_1\}\subseteq\omega^\omega$ is an
unbounded family such that $f_\alpha <^\ast f_\beta$ whenever $\alpha<\beta$.
Then $\cl(C)$ is not compact whenever $C\subseteq B$ is uncountable, where the
closure is taken in $\omega^\omega$.
\end{lemma}
\begin{proof}
Notice that every uncountable subset of $B$ is unbounded. Therefore, we can assume $C=B$ without loss of generality. Assume, in order to get a contradiction, that $\cl(B)$ is compact. Let $\pi_n:\omega^\omega\longrightarrow\omega$ be the projection on the $n$-th coordinate for every $n\in\omega$. Then $\pi_n[\cl(B)]$ is a compact subset of $\omega$ for each $n$. It follows that $B\subseteq\cl(B)\subseteq\prod_{n\in\omega}g(n)$ for some $g\in\omega^\omega$, which contradicts the fact that $B$ is unbounded.
\end{proof}

\begin{lemma}\label{overspill}
Assume that $B\subseteq\omega^\omega$ is an
unbounded family. Then $\cl(B)\nsubseteq\omega^\omega$, where the closure is taken in $(\omega+1)^\omega$.
\end{lemma}
\begin{proof}
Observe that $\cl(B)$ is compact because it is a closed subset of the compact space $(\omega+1)^\omega$. So $\cl(B)\subseteq\omega^\omega$ is impossible by Lemma \ref{noperfect}.
\end{proof}

\begin{lemma}\label{closureincreasing}
Assume that $B\subseteq\omega^\omega$ consists of
increasing functions. Then for every $f\in\cl(B)\setminus\omega^\omega$, where the closure is
taken in $(\omega+1)^\omega$, there exists $s\in\omega^{<\omega}$ such that $f=s^{\frown}\langle\omega,\omega,\ldots\rangle$.
\end{lemma}
\begin{proof}
Assume, in order to get a contradiction, that $f\in\cl(B)$ is such that $f(m)=\omega$ and $f(n)<\omega$, where $m<n$. Notice that
$$
U=\{x\in (\omega+1)^\omega: x(m)>f(n)\text{ and }x(n)=f(n)\}
$$
is a neighborhood of $f$ in $(\omega+1)^\omega$. However, the fact that $B$ consists of increasing functions implies that $U\cap B=\varnothing$, which is a contradiction.
\end{proof}

\begin{proposition}\label{bsmall}
Assume $\bbbb=\omega_1$. Then there exists a subspace $X$ of $2^\omega$ that has
the $\Pc$ but not the $\Pa$.
\end{proposition}
\begin{proof}
Since $\bbbb=\omega_1$, we can fix an unbounded family $B=\{f_\alpha:\alpha\in\omega_1\}\subseteq\omega^\omega$ such that $f_\alpha <^\ast f_\beta$ whenever $\alpha<\beta$. By Corollary \ref{existsgrinzing}, it is possible to fix $W\subseteq 2^\omega$ such that
$|W|=\omega_1$ and $W$ has the $\GP$. Let
$W=\{g_\alpha:\alpha\in\omega_1\}$ be an injective enumeration. Define
$Z=\{\langle f_\alpha,g_\alpha\rangle:\alpha\in\omega_1\}$.

Let $T=(\omega +1)^\omega\times 2^\omega$ and notice that $T\approx 2^\omega$.
For every $n\in\omega$, define
$$
T_n=\{x\in (\omega +1)^\omega:x(n)=\omega\}\times 2^\omega\subseteq T
$$
and observe that each $T_n\approx 2^\omega$. Let 
$$
X=Z\cup\bigcup_{n\in\omega}T_n\subseteq T.
$$
Notice that $Z$ is a $\mathsf{G_\delta}$ subset of $X$ because each $T_n$ is
compact.
Assume, in order to get a contradiction, that $Z$ contains a copy $K$ of
$2^\omega$.
Let $\pi:\omega^\omega\times 2^\omega\longrightarrow\omega^\omega$ be the
projection on the first coordinate. Then $\pi[K]$ is an uncountable
compact subset of $B$, because
$\pi\upharpoonright Z$ is injective and continuous. Since this contradicts Lemma
\ref{noperfect}, it follows that the space $X$ does not have the
$\Pg$. In particular, $X$ does not have the $\Pa$.

Observe that each $T_n$ has the $\Pc$ by Theorem \ref{suslin}. Therefore, in order to show that $X$ has the $\Pc$, it will be enough to prove that $\cl(Y)\cap \bigcup_{n\in\omega}T_n$ is uncountable for every uncountable $Y\subseteq Z$, where the closure is taken in $T$. Since $W$ has the $\GP$, it will be enough to show that $\cl(Y)\cap \bigcup_{n\in\omega}T_n$ is non-empty for every uncountable $Y\subseteq Z$. Notice that every uncountable subset of $B$ is unbounded. Therefore, we can assume that $Y=Z$ without loss of generality. By Lemma \ref{overspill}, there exists $f\in(\omega+1)^\omega\setminus\omega^\omega$ and a sequence $\langle \alpha_n:n\in\omega\rangle$ of elements of $\omega_1$ such that $\langle f_{\alpha_n}:n\in\omega\rangle$ converges to $f$ in $(\omega+1)^\omega$. On the
other hand, since $2^\omega$ is compact, there exists $g\in 2^\omega$ and a
subsequence of $\langle g_{\alpha_n}:n\in\omega\rangle$ that converges to $g$ in
$2^\omega$. It follows that the corresponding subsequence of $\langle \langle
f_{\alpha_n},g_{\alpha_n}\rangle:n\in\omega\rangle$ converges to $\langle
f,g\rangle$, which is clearly an element of $\bigcup_{n\in\omega}T_n$.
\end{proof}

One might wonder whether the factor $2^\omega$ in the definitions of $T$ and
$T_n$ is really needed.
In other words, if one defines $T=(\omega+1)^\omega$, $T_n=\{x\in
T:x(n)=\omega\}$ for $n\in\omega$, and $Z=B$,
does it follow that $\cl(Y)\cap \bigcup_{n\in\omega}T_n$ is uncountable
for every uncountable $Y\subseteq Z$? This first attempt was originally
suggested by
Kunen to the author, and
it obviously shaped the above proof. However, the answer is ``no'' in general. For
example, if $B$ consists of increasing functions then $\cl(Y)\cap \bigcup_{n\in\omega}T_n$ will always be contained in $T_\infty=\{s^{\frown}\langle\omega,\omega,
\ldots\rangle:s\in\omega^{<\omega}\}$ by Lemma \ref{closureincreasing}. In fact, this is the
very issue that led to the introduction of the Grinzing property.

We conclude this section by noting that the independence of the statement ``For every space $X$, if $X$ has the $\Pc$ then $X$ has the $\Pa$'' is much easier to prove than the sharper Theorem \ref{main}. In fact, it is enough to combine Proposition \ref{bbig} with the following result. We remark that, by a result of Judah and Shelah (see \cite{judahshelah}), the assumption $\bbbb=\omega_1$ (or even $\mathsf{unif(BP)}=\omega_1$) is not sufficient to guarantee the existence of a Luzin set.

\begin{proposition}[Zdomskyy]\label{luzin}
Assume that there exists a Luzin set in $\omega^\omega$. Then there exists a subspace $X$ of $2^\omega$ that has the $\Pc$ but not the $\Pa$.
\end{proposition}
\begin{proof}
This proof is very similar to that of Proposition \ref{bsmall}, so we will be brief. Define $T=(\omega+1)^\omega$ and $T_n=\{x\in
T:x(n)=\omega\}$ for $n\in\omega$. Let $Z=L$ be a Luzin set in $\omega^\omega$. Define $X=Z\cup\bigcup_{n\in\omega}T_n\subseteq T$. Notice that $Z$ witnesses that $X$ does not have the $\Pg$, because a Luzin set cannot contain any copy of $2^\omega$. In particular, $X$ does not have the $\Pa$.

In order to show that $X$ has the $\Pc$, it will be enough to prove that $\cl(Y)\cap \bigcup_{n\in\omega}T_n$ is uncountable for every uncountable $Y\subseteq Z$, where the closure is taken in $T$. So fix such a $Y$. Notice that $Y$ is not nowhere dense in $\omega^\omega$ because $Z$ is a Luzin set. Hence $U\subseteq\cl(Y)$, where $U=\{x\in\omega^\omega:s\subseteq x\}$ for some $s\in\omega^{<\omega}$. It is easy to check that $\cl(U)=\{x\in T:s\subseteq x\}$, which clearly has uncountable intersection with $\bigcup_{n\in\omega}T_n$.
\end{proof}

\section{The complete picture}

At this point, it seems natural to ask whether finer distinctions are possible.
For example, is it consistent that there exists a space with the
$\mathsf{PSP(Borel)}$ but not the $\Pa$? Is it consistent that there exists a
space with the $\Pg$ but not the
$\mathsf{PSP}(\mathbf{\Pi}^0_3)$? We only ask for consistency results because
Proposition \ref{bbig} shows that there are no such examples in $\ZFC$.

The following result of Solecki (see \cite[Theorem 1]{solecki})
shows that the answer to such questions (and several other variants of them) is ``no'' (see Corollary \ref{aneqgd}). When combined with Theorem \ref{main}, Proposition \ref{wo}, Proposition \ref{onedimbernstein}, and Proposition \ref{twodimbernstein}, this will allow us to obtain the ``complete solution'' promised in the introduction (see Theorem \ref{completepicture}).

\begin{theorem}[Solecki]\label{checkGdeltas}
Let $\II$ be a family of closed subsets of a Polish space $T$. Let $A$ be an
analytic subset of $T$. Then one of the following conditions holds.
\begin{enumerate}
\item\label{Asmall} $A\subseteq\bigcup\JJ$ for some countable $\JJ\subseteq\II$.
\item\label{Gbig} There exists a $\mathsf{G_\delta}$ subset $G$ of $T$ such that
$G\subseteq A$ and $G\nsubseteq\bigcup\JJ$ for every countable
$\JJ\subseteq\II$.
\end{enumerate}
\end{theorem}

\begin{corollary}\label{aneqgd}
Let $X$ be a space with the $\Pg$. Then $X$ has the $\Pa$.
\end{corollary}
\begin{proof}
Without loss of generality, assume that $X$ is a subspace of $T=[0,1]^\omega$.
Define
$$
\II=\{K\subseteq T: K\textrm{ is closed in }T\textrm{ and }|X\cap
K|<\omega_1\}.
$$
Let $A$ be an analytic subset of $T$ such that $A\cap X$ is uncountable. Since
$\omega_1$ has uncountable cofinality, condition $(\ref{Asmall})$ in Theorem
\ref{checkGdeltas} cannot hold. Therefore, condition $(\ref{Gbig})$ must hold
for some $\mathsf{G_\delta}$ subset $G$ of $T$. In particular $G\cap X$ is
uncountable, so it contains a copy of $2^\omega$ by the
$\Pg$. Since $G\subseteq A$, it follows that $A\cap X$
contains a copy of $2^\omega$,
which is what we needed to show.
\end{proof}

The following proposition answers a question that appeared in an earlier version of this article.
\begin{proposition}[Miller]\label{wo}
There exists a space $X$ that has the $\Pa$ but not the $\Pco$.
\end{proposition}
\begin{proof}
Let $\WO\subseteq\omega^\omega$ be the set of codes of well-orderings of $\omega$, as in \cite[page 485]{jech}. For each infinite $\alpha<\omega_1$, let $z_\alpha\in\WO$ be a code for a well-ordering of $\omega$ of type $\alpha$. Define $Z=\{z_\alpha:\omega\leq\alpha<\omega_1\}$ and $X=Z\cup(\omega^\omega\setminus\WO)\subseteq\omega^\omega$. By \cite[Lemma 25.9]{jech}, $\WO$ is a coanalytic subset of $\omega^\omega$, so $X\cap\WO=Z$ is an uncountable coanalytic subset of $X$. Since the Boundedness Lemma (see \cite[Corollary 25.14]{jech}) implies that the only analytic subsets of $\omega^\omega$ contained in $Z$ are the countable ones, $Z$ does not contain any copy of $2^\omega$. Therefore, $X$ does not have the $\Pco$.

In order to prove that $X$ has the $\Pa$, let $A$ be an analytic subset of $\omega^\omega$ such that $A\cap X$ is uncountable. Notice that $A\cap(\omega^\omega\setminus\WO)$ must be uncountable, otherwise $A\cap\WO=A\setminus (A\cap(\omega^\omega\setminus\WO))$ would be an uncountable analytic subset of $\omega^\omega$ that has uncountable intersection with $Z$, contradicting the Boundedness Lemma. Now it remains to apply Theorem \ref{suslin}.
\end{proof}

\begin{proposition}\label{onedimbernstein}
There exists a space $X$ that has the $\Po$ but not the $\Pc$.
\end{proposition}
\begin{proof}
Let $B$ be a Bernstein set in $[0,1]$. Define
$$
X=(B\times\{1\})\cup([0,1]\times [0,1))\subseteq [0,1]\times [0,1].
$$
Notice that $B$ is an uncountable closed subset of $X$ containing no copies of $2^\omega$. Therefore $X$ does not have the $\Pc$. To see that $X$ has the $\Po$, observe that every non-empty open subset of $X$ contains a copy of $[0,1]\times [0,1]$, hence a copy of $2^\omega$.
\end{proof}

The following proposition, which is essentially due to Medini and Milovich (see \cite[Theorem 28]{medinimilovich}), shows that the space given by Proposition \ref{onedimbernstein} can even be made zero-dimensional. Next, we recall some terminology that will be used in its proof. Every filter is assumed to be on $\omega$.
Furthermore, we will assume that $\COF\subseteq\FF\subsetneq\PP(\omega)$ for every filter $\FF$, where $\COF=\{x\subseteq\omega:|\omega\setminus x|<\omega\}$. Recall that $\CC\subseteq\PP(\omega)$ has the \emph{finite intersection property} if $\bigcap F$ is infinite whenever $F$ is a
non-empty finite subset of $\CC$. Given $x\subseteq\omega$, define $x^0=x$ and
$x^1=\omega\setminus x$. Recall that $\Aa\subseteq\PP(\omega)$ is an
\emph{independent family} if $\bigcap\{x^{\nu(x)}:x\in F\}$ is infinite whenever
$F$ is a non-empty finite subset of $\Aa$ and $\nu:F\longrightarrow 2$.

\begin{proposition}
There exists a subspace $X$ of $2^\omega$ that has the $\Po$ but not the $\Pc$.
\end{proposition}
\begin{proof}
Throughout this proof, we will identify $\PP(\omega)$ with $2^\omega$ via
characteristic functions.
In particular, we will identify every filter and every independent family with a
subspace of $2^\omega$. 

First, we will show that every filter $\FF\subseteq 2^\omega$ has the $\Po$.
This is trivial if
$\FF=\COF$, so assume that $\FF\supsetneq\COF$. Let $U$ be an uncountable open
subset of $\FF$. In particular $U\neq\varnothing$, so $[s]\cap\FF\subseteq U$
for some $s\in{}^{<\omega}2$. Now pick any coinfinite $z\in\FF$ such that
$z\upharpoonright\dom(s)=s$. It is easy to see that $\{x\in [s]:z\subseteq x\}$
is a copy of $2^\omega$ contained in $U$.

By \cite[Lemma 7]{medinimilovich}, there exists an independent family
$\Aa\subseteq 2^\omega$ that is homeomorphic to $2^\omega$.
In particular, we can fix a Bernstein set $B$ in $\Aa$. Since $\Aa$ is an
independent family, the collection
$$
\CC=B\cup\{\omega\setminus x:x\in\Aa\setminus B\}
$$
has the finite intersection property (actually, it is an independent
family). Therefore,
there exists a filter $\FF$ such that $\CC\subseteq\FF$. Let $X=\FF$. It is easy
to realize that
$B$ is an uncountable closed subset of $X$ that does not contain any copy of
$2^\omega$.
In particular, $X$ does not have the $\Pc$.
\end{proof}

\begin{proposition}\label{twodimbernstein}
There exists a space $X$ that has the $\Pcl$ but not the $\Po$.
\end{proposition}
\begin{proof}
Let $B$ be a Bernstein set in $[0,1]\times [0,1)$. Define
$$
X=B\cup([0,1]\times \{1\})\subseteq [0,1]\times [0,1].
$$
Notice that $B$ is an uncountable open subset of $X$ containing no copies of $2^\omega$. Therefore $X$ does not have the $\Po$.On the other hand, the fact that $B$ is connected (use the same argument as in
the proof of \cite[Lemma 3.2]{vanmills}) implies that $X$ is connected. Since
$X$ contains a copy of $[0,1]$, hence a copy of $2^\omega$, it follows that $X$
has the $\Pcl$.
\end{proof}
\noindent Notice that the space $X$ given by Proposition \ref{twodimbernstein} cannot be zero-dimensional because every open set in a zero-dimensional space can be written as a countable union of clopen sets.

\begin{theorem}\label{completepicture}
Consider the following conditions on a space $X$.
\begin{enumerate}
\item\label{coanalytic} $X$ has the $\Pco$.
\item\label{analytic} $X$ has the $\Pa$.
\item\label{Gdelta} $X$ has the $\Pg$.
\item\label{Fsigma} $X$ has the $\Pf$.
\item\label{closed} $X$ has the $\Pc$.
\item\label{open} $X$ has the $\Po$.
\item\label{clopen} $X$ has the $\Pcl$.
\end{enumerate}
The implications $(\ref{coanalytic})\rightarrow(\ref{analytic})\leftrightarrow (\ref{Gdelta})\rightarrow
(\ref{Fsigma})\leftrightarrow (\ref{closed})\rightarrow (\ref{open})\rightarrow
(\ref{clopen})$ hold for
every $X$. The implication $(\ref{Gdelta})\leftarrow (\ref{Fsigma})$ holds
for every $X$ if and only if $\bbbb>\omega_1$. There exist $\ZFC$
counterexamples to the implications $(\ref{coanalytic})\leftarrow(\ref{analytic})$, $(\ref{closed})\leftarrow (\ref{open})$, and
$(\ref{open})\leftarrow (\ref{clopen})$.
\end{theorem}
\begin{proof}
The implications $(\ref{Gdelta})\rightarrow (\ref{closed})$
and $(\ref{Fsigma})\rightarrow (\ref{open})\rightarrow (\ref{clopen})$ are trivial. The equivalence $(\ref{analytic})\leftrightarrow (\ref{Gdelta})$ follows from Corollary \ref{aneqgd}, and it shows that the implication $(\ref{coanalytic})\rightarrow(\ref{analytic})$ holds. The equivalence $(\ref{Fsigma})\leftrightarrow (\ref{closed})$ holds because $\omega_1$ has uncountable cofinality. This concludes the proof of the first statement. Now it is easy to realize that the second statement follows from Theorem \ref{main}. Finally, the third statement is the content of Proposition \ref{wo}, Proposition \ref{onedimbernstein}, and Proposition \ref{twodimbernstein}.
\end{proof}

\section{Generalizing the Grinzing property}

In this section, we will define a natural generalization of the Grinzing
property and extend some of the results from Section 2. While Question
\ref{grinzingq} below seems to be of independent interest, it will turn out to
be relevant in the next section as well (see the remark preceding Question
\ref{bsmallgen}).
\begin{definition}
Fix cardinals $\kappa,\lambda$ such that $\omega_1\leq\kappa\leq\cccc$ and
$\lambda\leq\kappa$. We will say that a subset $W$ of $2^\omega$ has the
\emph{$(\kappa,\lambda)$-Grinzing property} (briefly, the $\klGP$)
if $|W|\geq\kappa$ and for every $Y\subseteq W$ such that $|Y|\geq\kappa$ there
exist subsets $Y_\alpha$ of $Y$ for $\alpha\in\lambda$ such that
$|Y_\alpha|\geq\kappa$ for each $\alpha$ and $\cl(Y_\alpha)\cap
\cl(Y_\beta)=\varnothing$ whenever $\alpha\neq\beta$, where the closure is taken
in $2^\omega$.
\end{definition}
\noindent Notice that a subset $W$ of $2^\omega$ such that $|W|\geq\kappa$ has the $\klGP$ if and only if every subset of $W$ of size $\kappa$ has the $\klGP$. Also notice that the $\klGP$ gets stronger as $\lambda$ gets bigger. Furthermore, it is clear that the $(\omega_1,\omega_1)\text{-}\GP$ is simply the $\GP$.

Using the well-known Lemma \ref{findcrowded}, it is easy to show that $2^\omega$
has the $(\kappa,\omega)\text{-}\GP$ for every cardinal $\kappa\leq\cccc$ of
uncountable cofinality. The following proposition shows that the restriction on
the cofinality is really necessary.
\begin{proposition}\label{nogrinzing}
Let $\kappa$ be a cardinal of countable cofinality such that
$\omega_1<\kappa<\cccc$. Then no subset of $2^\omega$ has the
$(\kappa,2)\text{-}\GP$.
\end{proposition}
\begin{proof}
Fix cardinals of uncountable cofinality $\kappa_n<\kappa$ for $n\in\omega$ such
that $\kappa_n\to\kappa$. Let $W$ be a subset of $2^\omega$ such that
$|W|=\kappa$. Fix a compatible metric on $2^\omega$. Using Lemma
\ref{findcrowded}, it is possible to obtain subsets $Y_n$ of $W$ for
$n\in\omega$ such that $|Y_n|=\kappa_n$ and $\diam(Y_n)\to 0$ as $n\to\infty$.
Pick $z_n\in Y_n$ for each $n$. By compactness, there exists $z\in 2^\omega$ and
a subsequence of $\langle z_n:n\in\omega\rangle$ that converges to $z$. It is
easy to check that the corresponding subsequence of $\langle
Y_n:n\in\omega\rangle$ converges to $z$. Assume without loss
of generality that $\langle Y_n:n\in\omega\rangle$ converges to $z$. This implies $z\in\cl(Z)$ for every
$Z\subseteq Y=\bigcup_{n\in\omega}Y_n$ of size $\kappa$, where the closure is
taken in $2^\omega$. In particular, $Y$ witnesses that $W$ does not have the
$(\kappa,2)\text{-}\GP$.
\end{proof}
\begin{lemma}\label{findcrowded}
Let $\kappa$ be a cardinal of uncountable cofinality. Then every space of size
at least $\kappa$ has a $\kappa$-crowded subspace.
\end{lemma}
\begin{proof}
Let $X$ be a space of size at least $\kappa$. Define
$$
\UU=\{U\subseteq X: U\text{ is open in }X\text{ and }|U|<\kappa\}.
$$
and $V=\bigcup\UU$. It is obvious that $\UU$ is a cover of $V$. Let $\VV$ be a
countable subcover of $\UU$. Notice that $|V|<\kappa$ because $V=\bigcup\VV$ and
$\kappa$ has uncountable cofinality. So $Y=X\setminus V$ is non-empty.
Furthermore, it is clear that every non-empty open subset of $Y$ has size at
least $\kappa$.
\end{proof}

The following fundamental question is open. Notice that Corollary
\ref{existsgrinzing} gives a positive answer in the case
$(\kappa,\lambda)=(\omega_1,\omega_1)$.
\begin{question}\label{grinzingq}
For which cardinals $\kappa,\lambda$ such that
$\omega_1\leq\lambda\leq\kappa\leq\cccc$ and $\kappa$ has uncountable cofinality
is it possible to prove in $\ZFC$ that there exists a subset of $2^\omega$ with
the $\klGP$?
\end{question}

The following result (see \cite[Section 4]{millerm}), together with
Proposition \ref{gpch}, shows that the statement ``$2^\omega$ has the
$(\cccc,\cccc)\text{-}\GP$'' is independent of $\ZFC$.
\begin{theorem}[Miller]
It is consistent that for every $Y\subseteq 2^\omega$ of size $\cccc$ there
exists a continuous map $f:2^\omega\longrightarrow 2^\omega$ such that
$f[Y]=2^\omega$.
\end{theorem}
\begin{corollary}\label{consistentgrinzing}
It is consistent that $2^\omega$ has the $(\cccc,\cccc)\text{-}\GP$.
\end{corollary}
\begin{proof}
Use the fact that $2^\omega\approx 2^\omega\times 2^\omega$.
\end{proof}

On the other hand, the proofs of Proposition \ref{gpch} and Theorem \ref{gpma}
can easily be
adapted to obtain the following results.
\begin{proposition}\label{gpchgen}
Assume $\bbbb=\kappa$. Then $2^\omega$ does not have the
$(\kappa,\omega_1)\text{-}\GP$.
\end{proposition}
\begin{theorem}\label{gpmagen}
Assume $\MA$. Then $2^\omega$ has the $(\kappa,\kappa)\text{-}\GP$ for every
cardinal $\kappa<\cccc$ of uncountable cofinality.
\end{theorem}

\section{Generalizing the perfect set property}

We begin by stating the following natural generalization of the $\PSP$.
\begin{definition}
Fix an uncountable cardinal $\kappa$. Let $X$ be a space and $\bG$ a pointclass.
We will say that $X$ has the \emph{$\kappa$-perfect set property for $\bG$
subsets}
(briefly, the $\kPSP$) if for every $\bG$ subset $A$ of $X$ either $|A|<\kappa$
or $A$ contains a copy of $2^\omega$.
\end{definition}
\noindent Notice that the $\kPSP$ gets stronger as $\kappa$ gets smaller and as
$\bG$
gets bigger. Furthermore, it is clear that the
$\omega_1\text{-}\PSP$ is simply the $\PSP$.

The following is the most general question that we can imagine on this subject.
\begin{question}\label{verygeneral}
What is the status of the statement ``For every space $X$, if $X$ has the
$\kPSP$ then $X$ has the $\kappa'\text{-}\mathsf{PSP(\bG')}$'' as
$\kappa,\kappa'$ range over all uncountable cardinals and $\bG,\bG'$ range over
all pointclasses?
\end{question}
\noindent Theorem \ref{completepicture} can be viewed as a partial answer to the
above question,
in the case where $\kappa=\kappa'=\omega_1$ and $\bG,\bG'$ are at most
$\mathsf{analytic}$. 
In the rest of this section, we will essentially point out further concrete
instances of Question \ref{verygeneral} which seem particularly interesting.

The following two results generalize Proposition \ref{bbig} and Corollary
\ref{aneqgd} from $\omega_1$ to an arbitrary uncountable cardinal $\kappa$.
Lemma \ref{countablecof} will allow us to handle the case in which $\kappa$ has
countable cofinality.
\begin{proposition}\label{bbiggen}
Assume $\bbbb >\kappa$, where $\kappa$ is an uncountable cardinal. Then the
$\kappa\text{-}\Pc$ implies the
$\kappa\text{-}\Pa$ for every space.
\end{proposition}
\begin{proof}
If $\kappa$ has uncountable cofinality, the desired result follows from a
straightforward adaptation of the proof of Proposition \ref{bbig}. So assume
that $\kappa$ has countable cofinality, and let $X$ be a space with the
$\kappa\text{-}\Pc$. By Lemma \ref{countablecof}, there exists
a cardinal $\kappa'<\kappa$ of uncountable cofinality such that $X$ has the
$\kappa'\text{-}\Pc$. Since $\bbbb>\kappa'$, it follows that
$X$ has the $\kappa'\text{-}\Pa$, which obviously implies the
$\kappa\text{-}\Pa$.
\end{proof}

\begin{proposition}\label{soleckigen}
Let $\kappa$ be an uncountable cardinal. Then the
$\kappa\text{-}\Pg$ implies the
$\kappa\text{-}\Pa$ for every space.
\end{proposition}
\begin{proof}
If $\kappa$ has uncountable cofinality, the desired result follows from a
straightforward adaptation of the proof of Corollary \ref{aneqgd}. Otherwise, apply Lemma \ref{countablecof} as in the proof of Proposition \ref{bbiggen}.
\end{proof}

\begin{lemma}\label{countablecof}
Assume that $\bG$ is a pointclass such that $2^\omega$ has the
$\cccc\text{-}\PSP$. Let $\kappa$ be an uncountable cardinal of
countable cofinality. Assume that $X$ is a space with the
$\kappa\text{-}\PSP$. Then there exists a cardinal $\kappa'<\kappa$
of uncountable cofinality such that $X$ has the
$\kappa'\text{-}\PSP$.
\end{lemma}
\begin{proof}
Assume without loss of generality that $X$ is a subspace of $T=[0,1]^\omega$.
Fix a compatible metric on $T$. Fix cardinals of uncountable cofinality
$\kappa_n<\kappa$ for $n\in\omega$ such that $\kappa_n\to\kappa$. Assume, in
order to get a contradiction, that there exist $\bG$ subsets $A_n$ of $X$ for
$n\in\omega$ such that $|A_n|\geq\kappa_n$ for each $n$ and no $A_n$ contains a
copy of $2^\omega$. Define $A=\bigcup_{n\in\omega}A_n$. First assume that the
pointclass $\bG$ is closed under countable unions. Then $A$ is a $\bG$ subset of
$X$ of size at least $\kappa$, so it contains a copy $K$ of $2^\omega$ by the
$\kappa\text{-}\PSP$. Notice that $|K\cap A_n|=\cccc$ for some $n$
because $|K|=\cccc$ has uncountable cofinality. The fact that $K\approx
2^\omega$ has the $\cccc\text{-}\PSP$ concludes the proof in this
case.

Now assume $\bG=\mathbf{\Pi}^0_\xi$, where $\xi$ is a successor ordinal such
that $1\leq\xi <\omega_1$. Choose $A_n$ as above, but require in addition that $\diam(A_n)\to 0$ as $n\to\infty$. This is possible by Lemma
\ref{findcrowded}. Pick $z_n\in A_n$ for each $n$. By compactness, there exists
$z\in T$ and a subsequence of $\langle z_n:n\in\omega\rangle$ that converges to
$z$. It is easy to check that the corresponding subsequence of $\langle
A_n:n\in\omega\rangle$ converges to $z$. Assume without loss
of generality that $\langle A_n:n\in\omega\rangle$ converges to $z$. Define
$A=\bigcup_{n\in\omega}A_n$. In the case $\xi>1$, it is not hard to show that $A$ is a $\bG$ subset of $X$ of size at least $\kappa$, which yields a contradiction
as before. In the case $\xi=1$, consider $A\cup (\{z\}\cap X)$ instead.
\end{proof}

The following question is obviously inspired by Proposition \ref{bsmall}. By
inspecting its proof, it is not hard to realize that the answer would be ``yes''
assuming the existence of a subset of $2^\omega$ with the
$(\kappa,\omega_1)\text{-}\GP$ (see Question \ref{grinzingq}).
\begin{question}\label{bsmallgen}
Does $\bbbb=\kappa$ imply that there exists a space with the
$\kappa\text{-}\Pc$ but not the
$\kappa\text{-}\Pa$?
\end{question}
\noindent Notice that an affirmative answer to Question \ref{bsmallgen},
combined with Proposition \ref{bbiggen}, Proposition \ref{soleckigen} and Lemma \ref{countablecof}, would
allow us to generalize Theorem \ref{completepicture} from $\omega_1$ to an
arbitrary uncountable cardinal $\kappa$.

Finally, we observe that Corollary \ref{sierpinskicor} below shows that another concrete instance of
Question \ref{verygeneral} is settled by the following classical result (see \cite[Proposition 13.7]{kanamori}).
\begin{theorem}[Sierpi\'{n}ski]\label{sierpinski}
Every $\mathbf{\Sigma}^1_2$ subset of a Polish space $T$ can be written as
$\bigcup_{\alpha\in\omega_1} A_\alpha$, where each $A_\alpha$ is a Borel subset of $T$.
\end{theorem}
\begin{corollary}\label{sierpinskicor}
Assume that $X$ is a space with the $\omega_2\text{-}\Pg$. Then $X$ has the $\omega_2\text{-}\mathsf{PSP}(\mathbf{\Sigma}^1_2)$.
\end{corollary}
\begin{proof}
Let $A$ be a $\mathbf{\Sigma}^1_2$ subset of $X$ such that $|A|\geq\omega_2$. By Theorem \ref{sierpinski}, it is possible to write $A=\bigcup_{\alpha\in\omega_1} A_\alpha$, where each $A_\alpha$ is a Borel subset of $X$. Fix $\alpha$ such that $|A_\alpha|\geq\omega_2$. Notice that $X$ has the $\omega_2\text{-}\Pa$ by Proposition \ref{soleckigen}. In particular, $A_\alpha$ contains a copy of $2^\omega$.
\end{proof}

\section{Acknowledgements}

The author is very grateful to Arnie Miller for realizing that an early ``proof'' of Corollary \ref{existsgrinzing} used an assumption that is not provable in $\ZFC$, and for kindly ``donating'' to him its current proof and Proposition \ref{wo}. He also thanks Ken Kunen for contributing to the proof of Proposition \ref{bsmall}, as explained in the remark that follows it. Finally, he thanks Lyubomyr Zdomskyy for several discussions on the topic of this article, and in particular for suggesting Lemma \ref{closureincreasing} and the remarks on Luzin sets (namely, the second proof of Proposition \ref{gpch} and Proposition \ref{luzin}).

\end{document}